\pdfoutput=1
\RequirePackage{ifpdf}
\ifpdf 
\documentclass[pdftex]{sigma}
\else
\documentclass{sigma}
\fi

\numberwithin{equation}{section}

\newtheorem{Theorem}{Theorem}[section]
\newtheorem*{Theorem*}{Theorem}
\newtheorem{Corollary}[Theorem]{Corollary}
\newtheorem{Lemma}[Theorem]{Lemma}
\newtheorem{Proposition}[Theorem]{Proposition}
 { \theoremstyle{definition}
\newtheorem{Definition}[Theorem]{Definition}
\newtheorem*{Note}{Note}
\newtheorem{Example}[Theorem]{Example}
 }

\usepackage{stackengine}
\usepackage{scalerel}
\def\rlwd{.4pt}
\def\rlht{1.1pt}
\def\shatvrule{\rule{\rlwd}{\rlht}}
\def\shat#1{%
 \ThisStyle{%
 \setbox0=\hbox{$\SavedStyle#1$}%
 \stackon[0pt]{\stackon[1pt]{\ensuremath{\SavedStyle#1}}{%
 \shatvrule\kern\wd0\kern-\rlwd\kern-\rlwd\shatvrule}}%
 {\rule{\wd0}{\rlwd}}%
 }%
}

\usepackage{bm}

\DeclareMathOperator*{\clim}{c-lim}

\newcommand{\dd}[1]{\Delta^{^{(#1)}}}
\newcommand{\tdd}[1]{\widetilde{\Delta}^{^{(#1)}}}
\newcommand{\prol}[3]{\shat{#1}^{^{(#2)}}}
\newcommand{\proltd}[3]{\shat{\widetilde{#1}}^{^{(#2)}}}

\newcommand{\ulk}[3]{\shat{#1}^{^{(#2)}}_{#3}}
\newcommand{\dep}{\left.\frac{{\rm d}}{{\rm d}\varepsilon}\right|_{\varepsilon = 0}}

\newcommand{\ddx}{\frac{\Delta}{\Delta x}}

\newcommand{\ddf}[1]{\frac{\Delta}{\Delta x}\left[ #1 \right]}
\newcommand{\ddfn}[2]{\frac{\Delta^{^{(#2)}}}{\Delta x^{^{(#2)}}}\left[ #1 \right]}
\newcommand{\sftn}[2]{S^{#2}\left[ #1 \right]}
\newcommand{\sft}[1]{S\left[ #1 \right]}

\newcommand{\xilk}[3]{\prol{#1 x^{#2}\xi}{#3}{#3}}

\newcommand{\ev}[2]{\left. #1\right|_{#2}}

\usepackage{mathtools}

\begin{document}
\allowdisplaybreaks

\renewcommand{\thefootnote}{}

\newcommand{\arXivNumber}{2306.13175}

\renewcommand{\PaperNumber}{002}

\FirstPageHeading

\ShortArticleName{Computation of Infinitesimals for a Group Action on a~Multispace}

\ArticleName{Computation of Infinitesimals for a Group Action\\ on a Multispace of One Independent Variable\footnote{This paper is a~contribution to the Special Issue on Symmetry, Invariants, and their Applications in honor of Peter J.~Olver. The~full collection is available at \href{https://www.emis.de/journals/SIGMA/Olver.html}{https://www.emis.de/journals/SIGMA/Olver.html}}}

\Author{Peter ROCK}

\AuthorNameForHeading{P.~Rock}

\Address{Department of Mathematics, University of Colorado Boulder,\\
395 UCB, Boulder, CO~80309, USA}
\Email{\href{mailto:peter.rock@colorado.edu}{peter.rock@colorado.edu}}

\ArticleDates{Received July 04, 2023, in final form December 29, 2023; Published online January 02, 2024}

\Abstract{This paper expands upon the work of Peter Olver's paper [\textit{Appl. Algebra Engrg. Comm. Comput.} \textbf{11} (2001), 417--436], wherein Olver uses a~moving frames approach to examine the action of a group on a curve within a generalization of jet space known as multispace. Here we seek to further study group actions on the multispace of curves by computing the infinitesimals for a given action. For the most part, we proceed formally, and produce in the multispace a recursion relation that closely mimics the previously known prolongation recursion relations for infinitesimals of a group action on jet space.}

\Keywords{jet space; multispace; symmetry methods; differential equations; numerical or\-di\-nary differential equations}

\Classification{53A70; 58D19}

\renewcommand{\thefootnote}{\arabic{footnote}}
\setcounter{footnote}{0}

\section{Introduction}

In STEM fields, we are often interested in studying solutions to various types of differential equations
(DEs). To do this, we usually create computer models that predict the behavior of solutions to
particular DEs with prescribed boundary conditions. Then, these predictions can be compared to a known data set
and the resulting differences between the sampled and predicted data analyzed to help determine the correctness
of the model. While this has generally served us
well, the unfortunate truth is that the computers we use to construct such models are fundamentally limited.

Put bluntly, the problem comes down to one foundational question: how do we deal with imperfect
data for a solution to a DE? Up until this point, one of the most effective methods for studying solutions
to DEs
has been to study the set of symmetries of the solutions to the given DE under a given
group action. In fact, this was one of the favored approaches of Cartan for determining the existence of solutions to
DEs. More recently, many of Cartan's methods were modernized by Peter Olver and Mark Fels in their seminal
papers on moving frames (see~\cite{fels1998moving, fels1999moving}). So, it is then natural to ask what happens when we
try to adapt this modern approach to Cartan's methods to something a bit more discrete? This is the topic of~\cite{olver2001joint}, where the multispace of curves in a manifold is first defined.

The goal behind the construction of multispace was to generalize the notion of jet spaces. Specifically,
multispace was designed to
allow for the imposition of different types of contact conditions at multiple points along a solution curve.
Intuitively, this makes sense as an approximation method. For any set of contact conditions at any
finite set of points along a curve, we can construct a polynomial meeting those contact conditions. Thus
we can locally approximate the solution curve using the constructed interpolating polynomial curve. This
polynomial curve can then be used as the representative for the equivalence class of all smooth curves
meeting the desired contact conditions. The idea, then, is to study how groups act on these (multicontact)
equivalence classes of curves by studying how they act on their polynomial representatives. We then wish to
see if the symmetry relationships that we know for the smooth curves under a given group action can be
expanded in some meaningful way to their entire multicontact equivalence class.

In \cite{olver2001geometric}, Olver determined how to describe the action of the group on the entire
multispace and determined how to obtain what he termed ``multiinvariants'' for the action. He did not,
however, examine how to prolong such actions on equivalence classes of curves within multispace, and that is
what we seek to do in this paper. To this end, there are several natural questions to ask:
\begin{enumerate}\itemsep=0pt
 \item[(i)] How does the discrete analogue for the prolongation of a group action behave when it is restricted to curves?
 \item[(ii)] What are the infinitesimals for the infinitesimal generators of the action of a group on such a curve?
 \item[(iii)] How do these infinitesimals behave under the natural analogue of prolongation of a group action (i.e., do they satisfy a similar differential relation to the general prolongation formula as given in \cite[p.~110]{olver2000applications})? 
\end{enumerate}

We will answer all three of these questions in this paper, with the main theorem being that the answer to
question (iii) is a resounding ``Yes!''. The answer to the first question requires the construction of an
operator, denoted $\ddx$, that mimics the behavior of $\frac{{\rm d}}{{\rm d}x}$. More precisely, in terms of the
multispace coordinates $\bigl( x, u, \prol{u}{1}{1},\dots, \prol{u}{n}{n} \bigr)$ on $M^{(n)}$ (a
multispace of~$M$), we would like to have $\ddx\bigl[ \prol{u}{k}{k} \bigr] = \prol{u}{k + 1}{k + 1}$ just
as $\frac{{\rm d}}{{\rm d}x}\bigl[ u^{(k)} \bigr] = u^{(k + 1)}$ on~$J^{(n)}$ (the jet space of~$M$).
Then, after determining how group actions affect multispace coordinates, it will be possible to
show that the infinitesimals associated with an infinitesimal generator of a group action on a curve's
equivalence class in multispace obey a recursive relation. Moreover, this relation, detailed in Theorem~\ref{thm: miraculous_formula}, can be shown to closely mirror the recursive prolongation formula used to determine
the infinitesimals for a group action on the jet space of curves in a manifold.

\section{Background}

There are two things that require introduction before the main body of the paper: Olver's
multispace for curves, and the infinitesimals for the action of a Lie group on a manifold. The manifolds
that we will be most interested in for this paper are analogous to those that arise as the solutions to
differential equations. In other words, the results of this paper can be contextualized as examining the way
that a group action on a generalization of jet space -- multispace -- affects the graph of the solution of a
differential equation as it appears within that multispace.

\subsection{Jet space}

Since Olver's multispace of curves is designed to be an expansion on the idea of jet space, it seems prudent
to begin with a quick reminder of how jet space works. For the uninitiated, jet space is a tool which helps
us answer the following question: Given two curves in a manifold, how do we determine if these two curves
contain ``essentially the same'' differential information up to some order? The notion of equivalence that
is sought here is known as \textit{$k$-th order contact equivalence} and is captured in the following definitions:

\begin{Definition}
 Let $t$ be a coordinate on $\mathbb{R}$ and let $k\geq 0$. Two differentiable maps $f,g\colon \mathbb{R}\rightarrow \mathbb{R}$ with $f(0) = g(0) = 0$ are said to have ${k}$\textit{-th order contact} at $0$ if
 \[f^{(1)}(0) = g^{(1)}(0),\quad f^{(2)}(0) = g^{(2)}(0),\quad\dots,\quad f^{(k)}(0) = g^{(k)}(0),\]
 where $f^{(i)}(t)$ is the $i$-th derivative of $f$ with respect to $t$. Equivalently, the two maps
 $f$ and $g$ may be said to have the same \textit{$k$-jet at $0$}.
\end{Definition}

This notion of jet space can then be naturally extended to manifolds:

\begin{Definition}\label{def: jet_on_manifolds}
 Let $M$, $N$ be differentiable manifolds and $f,g\colon M\rightarrow N$ be two maps. Then $f$ and $g$ are said to have
 ${k}$\textit{-th order contact} at $p\in M$ if
 \begin{enumerate}\itemsep=0pt
 \item[(i)] $f(p) = g(p) = q$, and
 \item[(ii)] for all maps $u\colon \mathbb{R}\rightarrow M$ and $v\colon N\rightarrow \mathbb{R}$ with $u(0) = p$ and $v(q) = 0$, the differentiable maps $v\circ f \circ u$ and $v \circ g \circ u$ have the same $k$-jet at~$0$.
 \end{enumerate}
\end{Definition}

As a consequence of Definition~\ref{def: jet_on_manifolds}, we can then say that two differentiable maps $f,g\colon M\rightarrow N$
have the same $k$-jet at $p$ if and only if all the derivatives up to order $k$ with respect to the coordinates
in any pair of local coordinate systems around $p$ and $q$ agree. In more familiar terms, this
means that the Taylor series of~$f$ and~$g$ agree up to order~$k$ near the point~$p$.

\subsection{Polynomial interpolation}

A full discussion of Olver's multispace of curves also requires a brief review of the basics of polynomial
interpolation. It is well known that given a real-valued function of one variable $f$ and a
set of $n+1$ distinct points on the graph of $f$, $\{ (x_i,f_i)\}_{i = 0}^n$ (where $f_i =
f(x_i)$), there exists a unique degree $n$ polynomial
\[p(x) = a_0 + a_1x + \cdots + a_nx^n\]
such that $p(x_i) = f_i$ for each $i=0,\dots,n$. Such a polynomial is called an
\textit{interpolating polynomial of $f$ at the points $\{ (x_i,f_i)\}_{i = 0}^n$}, and
the basic idea behind the multispace of curves is that such a polynomial can then be used as a
sort of $n$-th order approximation for the function whose graph passes through the interpolated points.
This coincides with the familiar notion that the truncated $n$-th order Taylor series may be viewed as
an $n$-th order approximation to the function with graph passing through a single point.

To determine this interpolationg polynomial, we need only to solve the system of equations
\begin{gather*}
 a_0 + a_1x_0 + a_2x_0^2+\cdots + a_{n-1}x_0^{n-1} + a_nx_0^n = f_0,\\
 a_0 + a_1x_1 + a_2x_1^2+\cdots + a_{n-1}x_1^{n-1} + a_nx_1^n = f_1,\\
 \cdots \cdots \cdots \cdots \cdots \cdots \cdots \cdots \cdots \cdots \cdots \cdots \cdots \cdots \cdots \cdots\\
 a_0 + a_1x_n + a_2x_n^2+\cdots + a_{n-1}x_n^{n-1} + a_nx_n^n = f_n
\end{gather*}
for the coefficients $a_0,\dots, a_n$. This system may be rewritten as a matrix equation of the form $V\bm{a} = \bm{f}$ to obtain
\[\underbrace{\begin{pmatrix}
 1 & x_{0} & x^{2}_{0} & \cdots & x^{n}_{0} \\
 1 & x_{1} & x^{2}_{1} & \cdots & x^{n}_{1} \\
\vdots & \vdots & \vdots & \ddots & \vdots \\
 1 & x_{n} & x^{2}_{n} & \cdots & x^{n}_{n}
 \end{pmatrix}}_V\underbrace{\begin{pmatrix}
 a_0\\a_1\\\vdots\\a_n
 \end{pmatrix}}_{\bm{a}} = \underbrace{\begin{pmatrix}
 f_0\\f_1\\\vdots\\f_n
 \end{pmatrix}}_{\bm{f}}.\]
The matrix $V$ is called the \textit{Vandermonde matrix} or the \textit{sample matrix} for the interpolation problem. Whenever the points $x_i$ are distinct, $V$ is known to be invertible, so solving for the coefficients of the interpolating polynomial is as simple as multiplying by the inverse of $V$.

\subsection{The Olver multispace}

It is now possible to give a recap of Olver's description for the multispace of curves as it appears in
\cite{olver2001geometric}. Once again, the idea behind multispace is to extend the notion of the jet space in
a~natural way. In particular, we still wish study the equivalence classes of curves in a manifold up to some imposed
contact conditions at prescribed points. However, unlike in the jet space wherein we might consider the
$n$-th order contact condition for a function $u$ at a given point $x_0$, in the multispace it will instead
be possible to consider a set of contact conditions imposed at a collection of points $\{x_0,\dots,x_k\}$ up
to, what we will call, $n$-th order multicontact. What is meant by ``$n$-th order multicontact'' will need
to be expounded upon, but we will want for it to be consistent with the usual notion of contact equivalence.
Put succinctly, the notion of multicontact equivalence in multispace should to the usual notion of contact
equivalence on the jet space when $n$-th order multicontact on the singleton set $\{x_0\}$ is considered.

\begin{Definition}
 An \textit{$(n +1)$-pointed manifold} is an object $\bm{M} = (z_0,\dots, z_n; M)$ consisting of a~smooth manifold $M$ and $n + 1$ not necessarily distinct points $z_0,\dots, z_n\in M$. Given $\bm{M}$, we~let%
 \[n_i = \#\{j\colon u_j = u_i\}\]
 be the number of points that coincide with $u_i$.
\end{Definition}

\begin{Definition}
 Two $(n + 1)$-pointed curves $\bm{C}$ and $\widetilde{\bm{C}}$ in a manifold $M$ are said to have
 \textit{$n$-th order multi-contact} if and only if there exists a permutation $\sigma\colon \{0,\dots, n \}
 \rightarrow \{0,\dots, n \}$ such that%
 \[z_i = \widetilde{z}_{\sigma(i)}\qquad\text{and}\qquad \ev{\bm{j_{n_i - 1}{C}}}{z_i} = \ev{\bm{j_{n_i - 1}\widetilde{{C}}}}{z_{\sigma(i)}}\qquad\text{for each}\quad i = 0,\dots,n.\]
 where $\bm{j_k{C}}$ is the $k$-th jet of the curve $\bm{C}$.\footnote{This definition varies slightly
 from the one given in~\cite{olver2001geometric} and instead reflects the one given in \cite{mari2018discrete}.
 The main distinction is that the definition found in \cite{mari2018discrete} accounts for relabeling of the points.}
\end{Definition}

\begin{Definition}
 The \textit{$n$-th order multispace on curves in a manifold $M$}, denoted $M^{(n)}$, is the set of equivalence classes of $(n + 1)$-pointed curves in $M$ under the equivalence relation of $n$-th order multicontact. The equivalence class of an $(n + 1)$-pointed curve $\bm{C}$ is called an \textit{$n$-th order multijet}, and denoted $\bm{j_nC} \in M^{(n)}$.
\end{Definition}

Under this definition, when all the points on $\bm{C}$ are distinct, then $\bm{j_nC} =
\bm{j_n\widetilde{C}}$ if and only if~$\bm{C}$ and $\bm{\widetilde{C}}$ are coincident at all $n + 1$ points. On the other hand, if
$z_0 = \cdots = z_n$, then the curves will be equivalent if and only if they satisfy the same
$n$-th order contact condition at the point~$z_0$.

\begin{Definition}
 Let $z_0$ and $z_1$ be points on the two-pointed curve $\bm{C}$. Then the \textit{limit of $\bm{C}$ as~$z_1$ coalesces to $z_0$} (a.k.a.\ the \textit{coalescent limit} of $\bm{C}$) is defined by the limit of the
 sequence $(\widetilde{z}_i)_{i = 0}^\infty$ where $\widetilde{z}_0 = z_1$, $\widetilde{z}_i \rightarrow z_0$, and every
 $\widetilde{z}_i$ lies on~$\bm{C}$. Likewise, if the points $(z_i)_{i = 0}^n$
 are points on an~${n+1}$ pointed curve, then the coalescent limit as all of the points coalesce to the
 basepoint~$z_0$ is determined by taking the coalescent limit of each pair $z_0$ and $z_i$ for $1\leq i \leq n$.\footnote{It is known that the coalescent limit defined this way is independent of the order in which
 the points $z_i$ are chosen to coalesce to $z_0$ (cf.\ \cite{deBoordivided2005,olver2006multivariate}).}
\end{Definition}

\begin{Definition}
 Let $\bigl(x,u^1(x),\dots,u^{m-1}(x)\bigr)$ be a local coordinate expression for a curve $C$ on an $m$-manifold $M$. An \textit{$(n+1)$-pointed graph} consists of the graph of the smooth function $u(x)$ together with $(n + 1)$ not necessarily distinct points $z_i = (x_i, u_i)$ on the graph.
\end{Definition}

It is now possible to define the classical divided differences $[z_0,z_1,\dots,z_n]$ by the recursive rule
\[[z_0,z_1,\dots,z_{k - 1},z_k] = \frac{[z_0,z_1,\dots,z_{k-2},z_k] - [z_0,z_1,\dots,z_{k-2},z_{k-1}]}{x_k - x_{k-1}}, \qquad [z_j] = u_j.\]

In the case where $u$ is $\mathbb{R}^{m -1}$-valued, the divided differences can
be defined component-wise. When all of the points $x_i$ are distinct, these divided
differences are well defined, and may be denoted by $u[x_{i_{0}},\dots,x_{i_{k}}]$. If, however, the
$x_i$ are not distinct, we have the following.

\begin{Definition}
 Given an $(n + 1)$-pointed curve $\bm{C} = (z_0,\dots,z_n; C)$, its divided differences are defined by $[z_j]_C = u_j$, and
 \[[z_0,z_1,\dots,z_k]_C = \lim_{z\rightarrow z_k}\frac{[z_0,z_1,\dots,z_{k-2},z]_C - [z_0,z_1,\dots,z_{k - 1}]_C}{x - x_{k - 1}}.\]
 When taking the limit, the point $z = (x,u(x))$ \textit{must lie on the curve $\bm{C}$}, and take limiting values as $x\rightarrow x_k$ and $u(x) \rightarrow u_k$.
\end{Definition}

This allows for the determination of the coalescent limit along a curve as several points $(z_{i_1}, \dots,
z_{i_r})$ coalesce to a single basepoint $z_0$ by recursively applying the above definition. Such a limit will
be denoted using the notation
\[\clim_{(z_{i_1}, \dots, z_{i_r})\rightarrow z_0} \bm{C}.\]

\begin{Proposition}[\cite{olver2001geometric}]
 If $z_i = (x_i, u_i)$ where all of the $(x_{i})$ are distinct, then the unique interpolating polynomial at the points $z_0,\dots,z_n$ of degree $\leq n$ is given by
 \begin{gather*}
 p(x) = u[x_0] + u[x_0,x_1](x - x_0) + \cdots + u[x_0,\dots,x_n](x - x_0)\cdots(x - x_{n - 1}),\\
 \intertext{or, equivalently,}
 p(x) = [z_0] + [z_0,z_1](x - x_0) + \cdots + [z_0,\dots,z_n](x - x_0)\cdots(x - x_{n - 1}).
 \end{gather*}
 We will call this unique polynomial the $n$-th order Newton approximation of the function $u(x)$ based at the points $(x_i)$.
\end{Proposition}

\begin{Theorem}[\cite{olver2001geometric}]
 Two $(n + 1)$-pointed curves $\bm{C}$, $\widetilde{\bm{C}}$ have $n$-th order multicontact if and only if they have the same divided differences:
 \[[z_0,z_1,\dots,z_k]_C = [z_0,z_1,\dots,z_k]_{\widetilde{C}},\qquad k = 0,\dots,n.\]
\end{Theorem}

\subsection{Infinitesimals of group actions on jet space}

We now turn to the topic of infinitesimals of group actions on jet space. These infinitesimals provide
a natural way of studying subspaces of jet space under a group action by giving a way to (1)~determine what group actions leave the subspace invariant and (2)~transform the subspace
into a new (equivalent) subspace that is easier to work with. For more information on the particulars of
this process, the reader may refer to~\cite{olver2000applications}.

Consider the action of a Lie group $G$ on a manifold $M$. The orbits of the action of any
one-parameter subgroup $\gamma(\varepsilon)$ (with $\gamma(0) = e$) of $G$ on $M$ appear as the integral
curves of a~vector field $\bm{v}_\gamma$. This vector field is then called the \textit{infinitesimal generator} for the action of the
one-parameter subgroup. Indeed, if we take $\bm{y} = (y_1,\dots, y_m)$ to be coordinates on~$M$, then the
infinitesimal generator may be locally expressed as
\[\bm{v}_\gamma = \sum_{i = 1}^m \zeta_{i,\gamma}(\bm{y}) \frac{\partial}{\partial y_i},\]
where the symbols $\zeta_{i,\gamma}(\bm{y})$ will be called the \textit{infinitesimals} for the one-parameter
group action. The infinitesimal generators that are of interest to us in this work are
those arising from the action of a one-parameter subgroup of $G$ on a manifold $M^2$ with local coordinates $(x, u) $.\footnote{The definition for $M^n$ with coordinates
$(\bm{x},\bm{u})$ is analogous, but gets a bit notationally cumbersome.} In this case, for
the one-parameter subgroup, $\gamma(\varepsilon)$, the infinitesimal generator takes the form
\[\bm{v}_\gamma = \xi_{\gamma}(x,u)\frac{\partial}{\partial x} + \varphi_{\gamma}(x,u) \frac{\partial}{\partial u}.\]
When we consider the action on a curve in $J^{(0)}(\mathbb{R}, \mathbb{R})=M$ representing the solution
to the DE $F\bigl(x,u,u^{(1)},\dots, u^{(n)}\bigr) = 0$, it is possible to see
that the ODE describes an embedded submanifold of the larger jet space $J^{(n)}$ with
local coordinates $\bigl(x, u, u^{(1)},\dots, u^{(n)}\bigr)$. If we were to then consider the infinitesimal generator
of the induced one-parameter group action on $J^{(n)}(\mathbb{R},\mathbb{R})$, it would have an expression of the form
\begin{gather*}
\bm{v}_\gamma^{(n)} = \xi_{\gamma}(x,u)\frac{\partial}{\partial x} + \varphi_{\gamma}(x,u) \frac{\partial}{\partial u} + \varphi_{[1],\gamma}\bigl(x, u, u^{(1)}\bigr) \frac{\partial}{\partial u^{(1)}} + \cdots\\
\hphantom{\bm{v}_\gamma^{(n)} =}{} + \varphi_{[n],\gamma}\bigl(x, u, u^{(1)},\dots, u^{(n)}\bigr) \frac{\partial}{\partial u^{(n)}}.
\end{gather*}
Here the infinitesimals are obtained from the induced group action on $J^{{(n)}}(\mathbb{R},\mathbb{R})$ using the standard method for
change of coordinates. More specifically, if
\[
(\widetilde{x}(x,u;\varepsilon), \widetilde{u}(x,u;\varepsilon)) =
\gamma(\varepsilon)\cdot (x, u),
\]
 then
\begin{gather*}
 \xi_{\gamma}(x,u) = \ev{\frac{{\rm d}}{{\rm d}\varepsilon}}{\varepsilon = 0} \left[ \widetilde{x} \right],\qquad
 \varphi_{\gamma}(x,u) = \ev{\frac{{\rm d}}{{\rm d}\varepsilon}}{\varepsilon = 0} \left[ \widetilde{u} \right],\qquad
 \varphi_{[1], \gamma}(x,u) = \ev{\frac{{\rm d}}{{\rm d}\varepsilon}}{\varepsilon = 0} \left[\frac{{\rm d}\widetilde{u}}{{\rm d}\widetilde{x}} \right],\\
 \varphi_{[2], \gamma}(x,u) = \ev{\frac{{\rm d}}{{\rm d}\varepsilon}}{\varepsilon = 0} \left[\frac{{\rm d}^2\widetilde{u}}{{\rm d}\widetilde{x}^2} \right],\qquad\dots.
\end{gather*}

This new vector field, $\bm{v}_\gamma^{(n)}$, with infinitesimals as given above, is then known as the \textit{$n$-th prolongation} of the vector field $\bm{v}_\gamma$, and it serves as the infinitesimal generator of the \textit{prolonged group action on the jet space}.
We may also sometimes refer to the prolongation of~$M$ without referring to any particular group $G$. In this case, we are simply referring to the natural extension from $M$ with local coordinates $(x,u)$ to $J^{(n)}$ with local coordinates $\bigl( x, u, u^{(1)},\dots,u^{(n)} \bigr)$ obtained by taking derivatives.

\begin{Note}
For ease of notation, the subscript $\gamma$ will be dropped in the rest of this paper unless it is otherwise necessary.
\end{Note}

The reader should also note here that there is a difference between the $k$-th derivative of $\varphi$ with respect to $x$, which would be denoted by $\varphi^{(k)}$, and the infinitesimal for the prolongation of the $k$-th derivative, which is denoted by $\varphi_{[k]}$. There is, however, a recursive relationship between these two objects, as is given in the following proposition.

\begin{Proposition}[\cite{mansfield2010practical, olver2001geometric}]
 \begin{equation}\label{eq: 1d-recursion-smooth}
 \varphi_{[k]} = \frac{{\rm d}}{{\rm d}x}\left[ \varphi_{[k-1]} \right] - u^{(k)}\frac{{\rm d}}{{\rm d}x}\xi,
 \end{equation}
 where $u^{(k)}$ is the $k$-th derivative of $u$, and $\frac{{\rm d}}{{\rm d}x}$ is the total derivative. Equivalently,
 this relation may be stated as
 \begin{equation}\label{eq: 1d-rec-expansion}
 \varphi_{[k]} = \frac{{\rm d}^k}{({\rm d}x)^k}\varphi - \sum\limits_{i = 1}^{k}\binom{k}{i-1}\frac{{\rm d}^i u}{({\rm d}x)^i} \frac{{\rm d}^{k - i + 1}\xi}{({\rm d}x)^{k - i + 1}},
 \end{equation}
 or as
 \begin{equation}\label{eq: useful-derivative}
 \varphi_{[k]} = \frac{{\rm d}^k}{({\rm d}x)^k}\bigl[ \varphi - u^{(1)}\xi \bigr] + u^{(k + 1)}\xi,
 \end{equation}
 where $\frac{{\rm d}^k}{({\rm d}x)^k}$ is the $k$-th total derivative. This third equation is also known as the
 \textit{characteristic form} of the prolongation formula and $\varphi - u^{(1)}\xi$ is known as the
 \textit{characteristic} of $\bm{v}$.
\end{Proposition}

The (arguably) most important application of the characteristic formula can be
found in the proof of the Noether theorems for variational symmetries and conservation laws. With this in
mind, it will be the goal of this paper to reconstruct the above formulae, or something similar to them, in the case
of multispace.

\section{Differential structure}

As opposed to the jet space prolongation where we were able to jump directly into the computation of
the infinitesimals for a group action, some issues arise when we attempt to define the appropriate
analogue of prolongation for $n$-pointed curves in multispace. Differentiating, for example,
can be a bit tricky. As such, it is first necessary to build up the differential structure for
general divided differences before considering what it means to prolong a~group action on a~curve in
multispace.

Whereas in the traditional jet space the lift of the $n$-jet at a given point along a~curve to an $(n + 1)$-jet is uniquely determined, the notion of ``multi-contact'' that allows for the
construction of polynomial approximations of curves makes it so that the lift is no longer unique.
Indeed, we can extend the $n$-th order contact condition at a point $(x_0, u_0)$ along the curve to
either an $(n + 1)$-st order contact condition at that point, or to a $n$-th order contact condition at that
point and a transversality condition at some other point $(x_1, u_1)$. Thus, a methodology which tracks
this extra information for the analogue of prolongation in the multispace is required.

\subsection{Notation and divided differences}

We begin by giving a presentation of the classical divided difference formulas that will be more useful
throughout the rest of the paper. For simplicity, we will restrict to the case of scalar-valued $u$.\footnote{The
formulas presented here can be easily extended to vector-valued $u$ using component-wise application of the presented
formulas} Let $U$ be an open subset of $\mathbb{R}$; then for $x_0,x_1 \in U$, we will let $\dd{1}$ denote the
determinant of the Vandermonde matrix given by
\[\dd{1} = \left| \begin{matrix*}
 1 & x_0\\ 1 & x_1
\end{matrix*} \right|,\]
and for $x_0,x_1,x_2 \in U$, we will let
\[\dd{2} = \left| \begin{matrix*}
 1 & x_0 & x_0^2\\ 1 & x_1 & x_1^2\\ 1 & x_2 & x_2^2
\end{matrix*} \right|,\]
and so on.\footnote{The idea with this notation is that the superscript will eventually denote the order of the derivative approximation.}

\begin{Definition}
 We say that the set $\Gamma = \{ x_i \in U\colon i = 0,\dots,n \}$ is a \textit{multispace lattice of order $n$ and basepoint $x_0$} if $\dd{n} \neq 0$, i.e., if the $x_i$ are all distinct.
\end{Definition}

To prove Proposition~\ref{prop: discrete_action_computation}, it is also necessary to introduce the notation $\dd{k}\bigl(u,
x^\ell\bigr)$ to represent the determinant wherein the column corresponding to $x^\ell$ is replaced with the
values of $u$ evaluated at each of the $x_i$. So, for example, $\dd{4}\bigl(u,x^2\bigr)$ will denote the matrix
\[\dd{4}\bigl(u,x^2\bigr) = \left| \begin{matrix*}
 1 & x_0 & u_0 & x_0^3 & x_0^4\\
 1 & x_1 & u_1 & x_1^3 & x_1^4\\
 1 & x_2 & u_2 & x_2^3 & x_2^4\\
 1 & x_3 & u_3 & x_3^3 & x_3^4\\
 1 & x_4 & u_4 & x_4^3 & x_4^4
\end{matrix*} \right|.\]

\begin{Proposition}[\cite{kellison1975numerical}]\label{prop: lagrange_approx}
 Given $u\colon U\rightarrow \mathbb{R}$ and $n + 1$ distinct points $(x_{0},\dots,x_{n})\in U$, the $n$-th order Newton approximation for the points $\{(x_i, u_i)\}$ is given by
 \begin{gather*}
 p(x) = u_0 + \frac{\dd{1}(u, x)}{\dd{1}}(x - x_0) + \frac{\dd{2}\bigl(u, x^2\bigr)}{\dd{2}}(x - x_0)(x - x_1) + \cdots\\
 \hphantom{p(x) =}{}
 + \frac{\dd{n}(u,x^n)}{\dd{n}}(x-x_0)\cdots(x-x_{n - 1}).
\end{gather*}
\end{Proposition}

\begin{Corollary}[\cite{olver2001geometric}]
 \[\clim_{(x_0,\dots,x_k)\rightarrow x_0} \frac{\dd{k}\bigl(u,x^k\bigr)}{\dd{k}} = \frac{u^{(k)}(x_0)}{k!}.\]
\end{Corollary}

As is shown in \cite{olver2001geometric}, the coalescent limit for divided difference formulas is well defined and independent of the order in which pairs of points are allowed to coalesce. And so we will introduce the general notation
\[\prol{u}{k}{k} := k!\frac{\dd{k}\bigl(u,x^k\bigr)}{\dd{k}},\]
where the upper parenthetical index is used to keep track of the number of points used to construct the
derivative approximation (in this case, $k + 1$ points). In terms of this notation, a~multispace~$M^{(n)}$
can then be given local coordinates $\bigl(x_0,\dots,x_n,\prol{u}{0}{0},\prol{u}{1}{1},\dots,
\prol{u}{n}{n} \bigr)$ analogous to the coordinates on $J^{(n)}$ given by $\bigl( x, u, u^{(1)}, \dots,
u^{(n)} \bigr)$. This notation will also sometimes be extended to keep track of the basepoint of the lattice.
So the notation $\prol{u}{k}{k}(x_i)$ will refer to the divided difference formula on the lattice $\Gamma_i =
 \{ x_i,\dots, x_{i + k} \}$.

This notation also provides a familiar way of expressing the formula in Proposition~\ref{prop: lagrange_approx}:
\[p(x) = \frac{\prol{u}{0}{0}}{0!} + \frac{\prol{u}{1}{1}}{1!}(x - x_0) + \frac{\prol{u}{2}{2}}{2!}(x - x_0)(x - x_1) + \cdots + \frac{\prol{u}{n}{n}}{n!}(x-x_0)\cdots(x - x_{n - 1}).\]
When we compare this to the Taylor series for $u$, we arrive at the convenient formula
\[\clim_{(x_0,\dots,x_k)\rightarrow x_0}\prol{u}{k}{k} = u^{(k)}(x_0).\]

We can now state more precisely what we mean by the \textit{prolongation of a curve} in $M$.

\begin{Definition}
 Given a 1-pointed curve in $M$ consisting of the graph of a smooth function $u(x)$ of~1 variable
 together with a point $(x_0, u_0)$ on the graph, the \textit{$n$-th order multispace prolongation of $u$
 at the $($not necessarily distinct$)$ points $(x_0,\dots, x_n)$}, is given by evaluating the (possibly
 coalesced) divided differences $\prol{u}{k}{}$ of $u$ for $1 \leq k \leq n$. The points used for the
 divided differences are assumed to be taken using the ordering imposed by $(x_0, \dots, x_n)$. When the
 points $(x_0,\dots, x_n)$ are understood, we will commonly refer to this as the \textit{$n$-th prolongation
 of~$u$.}

 Likewise, when we wish to extend an $n$-th order prolongation of $u$ at the points $(x_0,\dots,x_n)$ to
 an $(n + 1)$-st order prolongation at the points $(x_0,\dots, x_n, x_{n + 1})$, this may be done by adding a
 point and computing the appropriate divided differences. We will commonly refer to this process of extending
 the order by 1 as \textit{prolonging} the $n$-pointed curve.
\end{Definition}

This definition preserves the usual projection maps from jet space. If we consider the
$n$-th prolongation of $u$, $M^{{(n)}}_{(x_0,\dots,x_n)}$, then
$\pi\colon M^{{(n)}}_{(x_0,\dots,x_n)}\rightarrow M^{{(n-1)}}_{(x_0,\dots,x_{n-1})}$ is well defined and reduces
to the usual projection map on jet spaces $\pi\colon J^{(n)}_{x_0}\rightarrow J^{(n - 1)}_{x_0}$ when $x_0 = x_1 =
\dots= x_n$. The key thing to note about this mapping is that, unlike in the jet case, the extension from
$M^{(n-1)}$ to~$M^{(n)}$ is not unique: it depends on the choice of the point $x_n$. This is something that
to be keenly aware of as we construct $\ddx$, the appropriate analogue for $\frac{{\rm d}}{{\rm d}x}$.

\subsection{Derivatives of discrete equations}

Before the infinitesimals can be computed, we will need to devise an operator that takes the
approximation for the $k$-th derivative of a function $u$ at a given point, $x_0$, denoted $\prol{u}{k}{k}(x_0)$, to
the approximation for the $(k + 1)$-st derivative. That is to say, we desire some operator$\ddx$
that mimics the smooth operator $\frac{{\rm d}}{{\rm d}x}$ in the sense that
\[
\frac{\Delta}{\Delta x}\bigl[\prol{u}{k}{k} \bigr]
= \prol{u}{k+1}{k+1}.\]

Some examination of the recursive definitions for the divided difference formulae shows that the appropriate
definition for the discrete derivative of $\prol{u}{k}{k}$ at the basepoint $x_r$ may be given by
\[
\frac{\Delta}{\Delta x}\bigl[\prol{u}{k}{k}(x_r) \bigr]
= \frac{k+1}{x_{k + r + 1} - x_{r}}(S-{\rm id})\bigl[ \prol{u}{k}{k}(x_r) \bigr] = \prol{u}{k+1}{r+1}(x_r),\]
where $S$ is the shift operator $S[u_i] = u_{i + 1}$.

The natural extension of this operator to a higher-order derivative $\frac{{\rm d}^n}{{\rm d}x^n}$ can then be obtained
by iteratively applying the definition
\begin{gather*}
\frac{\Delta^{^{(n)}}}{\Delta x^{^{(n)}}}\bigl[ \prol{u}{k}{k}(x_r) \bigr]
:= \frac{k + n}{x_{k + n + r} - x_{r}}(S-{\rm id})\\
\hphantom{\frac{\Delta^{^{(n)}}}{\Delta x^{^{(n)}}}\bigl[ \prol{u}{k}{k}(x_r) \bigr]
:=}{}\times
\left[ \frac{k+n-1}{x_{k + n + r - 1} - x_{r}}(S-{\rm id})\left[ \cdots\frac{k+1}{x_{k + r + 1} - x_{r}}(S-{\rm id})\bigl[ \prol{u}{k}{k}(x_r) \bigr] \cdots\right]\right] \\
\hphantom{\frac{\Delta^{^{(n)}}}{\Delta x^{^{(n)}}}\bigl[ \prol{u}{k}{k}(x_r) \bigr]}{}
\ = \prol{u}{k + n}{k + n}(x_r).
\end{gather*}

This definition may then be broadened to any expression $u_{r,\dots,k+r}$ depending on the points $\{ x_r,\dots, x_{k + r} \}$:
\[\ddf{u_{r,\dots,k+r}} = \frac{k+1}{x_{k + r + 1} - x_{r}}(S-{\rm id})\left[ u_{r,\dots,k+r}\right].\]

\begin{Proposition}\label{prop: div-rules}
 Let $\left\{ \Gamma_r \right\}$ be a sequence of lattices with $\Gamma_r = \left\{ x_r,\dots, x_{k + r} \right\}$ and $\Gamma_r\cap\Gamma_{r+1} = \left\{ x_{r+1},\dots, x_{k + r} \right\}$. Suppose that $u_{r,\dots,k+r}$ and $v_{r,\dots,k+r}$ are functions defined on $\Gamma_r$. Then the following properties hold:
 \begin{enumerate}\itemsep=0pt
 \item[$1.$] Formal commutation with shifts:
 \[\ev{\ddf{\sft{u_{r,\dots,k+r}}}}{\Gamma_{r+1}} = \sft{\ev{{\ddf{u_{r,\dots,k+r}}}}{\Gamma_r}},\]
 where the shift is assumed to modify the basepoint from $x_r$ to $x_{r + 1}$. That is, if $u_{r,\dots,k+r}$ is defined on a lattice $\Gamma_r$, then $($assuming that $\Gamma_{r + 1}$ is defined$)$ $S[ u_{r,\dots,k+r}]$ is the same algebraic expression, but defined on the lattice $\Gamma_{k + 1}$ given by incrementing the subscripts by~$1$.

 \item[$2.$] Product rule:
 \[\ddf{(uv)_{r,\dots,k+r}} = \ddf{u_{r,\dots,k+r}}v_{r,\dots,k+r} + \sft{u_{r,\dots,k+r}}\ddf{v_{r,\dots,k+r}}.\]
 \item[$3.$] Quotient rule:
 \[\ddf{\frac{u_{r,\dots,k+r}}{v_{r,\dots,k+r}}} = \frac{v_{r,\dots,k+r}\ddf{u_{r,\dots,k+r}} - u_{r,\dots,k+r}\ddf{v_{r,\dots,k+r}}}{v_{r,\dots,k+r}\cdot \sft{v_{r,\dots,k+r}}}.\]
 \end{enumerate}
\end{Proposition}

\begin{proof}
1.~The formal commutation with shifts arises from a simple application of the definitions as follows:
 \begin{align*}
 \ev{\ddf{\sft{u_{r,\dots,k+r}}}}{\Gamma_{r+1}} &= \ddf{{u_{r+1,\dots,k+r+1}}}
 = \frac{k+1}{x_{k + r + 2} - x_{r+1}}(S-{\rm id}) [ {u_{r+1,\dots,r+k+1}} ]\\
 &= \sft{\frac{k+1}{x_{k + r + 1} - x_{r}}(S - {\rm id}) [ u_{r,\dots,k+r} ]}
 =\sft{\ev{{\ddf{u_{r,\dots,k+r}}}}{\Gamma_r}}.
 \end{align*}

2.~To obtain the product rule, we apply the definition of our differential operator once more and
 trace through the algebra:
 \begin{align*}
 \ddf{(uv)_{r,\dots,k+r}}={}& \frac{k+1}{x_{k + r + 1} - x_{r}}(S - {\rm id}) [ (uv)_{r,\dots,k+r} ]\\
={}& \frac{k+1}{x_{k + r + 1} - x_{r}} [ (uv)_{r+1,\dots,k+r+1} - (uv)_{r,\dots,k+r} ]\\
={}& \frac{k+1}{x_{k + r + 1} - x_{r}} [ (uv)_{r+1,\dots,k+r+1} - u_{r+1,\dots,k+r+1}v_{r,\dots,k+r} ]\\
& + \frac{k+1}{x_{k + r + 1} - x_{r}} [u_{r+1,\dots,k+r+1}v_{r,\dots,k+r} - (uv)_{r,\dots,k+r} ]\\
={}& \sft{u_{r,\dots,k+r}}\ddf{v_{r,\dots,k+r}} + \ddf{u_{r,\dots,k+r}}v_{r,\dots,k+r}.
 \end{align*}

3.~Lastly, for the quotient rule, we once more apply the definition of the operator and follow the
 algebra. Note that some special care needs to be taken to make sure that the shift operator is tracked appropriately:
 \begin{align*}
 \ddf{\frac{u_{r,\dots,k+r}}{v_{r,\dots,k+r}}}={}& \frac{k+1}{x_{k + r + 1} - x_{r}}(S - {\rm id}) [ \frac{u_{r,\dots,k+r}}{v_{r,\dots,k+r}} ]\\
={}&\frac{k+1}{x_{k + r + 1} - x_{r}}\left[ \frac{u_{r+1,\dots,r+1+k}}{v_{r+1,\dots,r+1+k}} - \frac{u_{r,\dots,k+r}}{v_{r,\dots,k+r}} \right]\\
 ={}& \frac{k+1}{x_{k + r + 1} - x_{r}}\left[ \frac{u_{r+1,\dots,r+1+k}v_{r,\dots,k+r} - u_{r,\dots,k+r}v_{r+1,\dots,r+1+k}}{v_{r,\dots,k+r}v_{r+1,\dots,r+1+k}}\right]\\
={}& \frac{k+1}{(x_{k + r + 1} - x_{r})v_{r,\dots,k+r}v_{r+1,\dots,r+1+k}} [ (u_{r+1,\dots,r+1+k} - u_{r,\dots,k+r})v_{r,\dots,k+r} \\
& + u_{r,\dots,k+r}(v_{r,\dots,k+r} - v_{r+1,\dots,r+1+k}) ] \\
={}& \frac{v_{r,\dots,k+r}\ddf{u_{r,\dots,k+r}} - u_{r,\dots,k+r}\ddf{v_{r,\dots,k+r}}}{v_{r,\dots,k+r}\cdot \sft{v_{r,\dots,k+r}}}.\tag*{\qed}
 \end{align*}\renewcommand{\qed}{}
\end{proof}

\begin{Corollary}[\cite{deBoordivided2005,Tibshirani2022}]\label{cor: prod-rule}
 If $u$ and $v$ are both defined on a lattice with basepoint $x_0$, then
 \[\ddfn{(uv)_{0,\dots,n}}{n} = \sum_{k = 0}^{n}\binom{n}{k} \sftn{\ddfn{u_{0,\dots,n}}{n-k}}{k}\ddfn{v_{0,\dots,n}}{k}.\]
\end{Corollary}

Of course, we may not always desire to take the discrete derivative of a product where each term in the product depends on exactly the same number of sample points (as in the above proposition). Indeed, as it currently stands, the following derivative is not well defined when $k \neq l$:
\begin{equation*}
 \ddf{u_{r,\dots,k + r}v_{r,\dots,l + r}}.
\end{equation*}

We can fix this problem if we take our lattices to be evenly spaced so that $x_i = x_0 + ih$ for some $h > 0$.
Then we have
\[\frac{k+1}{x_{k + 1} - x_{0}}=\frac{l+1}{x_{l + 1} - x_{0}}=\frac{k + l +1}{x_{k + l + 1} - x_{0}} = \frac{1}{h}.\]

However, lattices with such restrictions have already been studied rather extensively (see
\cite{dorodnitsyn2001307,Dorodnitsyn_2022, mansfield2019moving1,mansfield2019moving2}), so it would be a bit
redundant to cover it again here. In addition, in order to work on the multispace, \textit{we must allow for
any pair of points to coalesce}, and so restricting to only evenly spaced lattices will not allow us to
generalize the method for prolongation of the group action in the way that we desire.

\section{Computation of the infinitesimals}

Now that the differential structure for general divided differences has been formally defined, it
is possible to move on to the computation of the infinitesimals for group action on multispace.
Special care needs to be taken in the computation of these infinitesimals, however. Not only do the
infinitesimals need to be well defined, but they also need to
behave in a way that is consistent with the prolongation formulae for
jet space. Then, so long as the full coalescent limit is defined, it will be defined for any subset of
points (hence, the entire multispace), and we will be able to obtain a formula for the action of the group
on a prolongation of any multispace curve. In this way, our general strategy for this section will be to
begin with finding the prolongation formulae for group actions on the multispace coordinates, and then we
may check to see that as points coalesce down to a single basepoint, we obtain equations~\eqref{eq:
useful-derivative} and~\eqref{eq: 1d-recursion-smooth}.

Within the first section, we will assume that all of our sample points are distinct so that we may obtain
the recursive formulae, but we shall find in the second section that the coalescent limit is well behaved.
As such, to obtain recursive formulas for sets of non-distinct points, it is enough to start with the
general formula for distinct points and then apply the coalescent limit.

\subsection{Infinitesimal actions}

The first thing that needs to be established is what it means for a Lie group to act on multispace.
Given some set of points $(z_i) = (x_i, u_i)$ and an action of $G$ on $M$, the action of an element
$g\in G$ on this set produces a new set of points $(\widetilde{z}_i) = ((\widetilde{x}_i,\widetilde{u}_i)) = g\cdot
((x_i,u_i)) = g\cdot (z_i)$\footnote{The reader should note that, under the action $\widetilde{x_i}$ and $\widetilde{u}_i$
may depend on both $\vec{x}$ and $\vec{u}$.} to which we apply the definition of the multispace
coordinates. In terms of Vandermonde matrices, we may then write
\[g\cdot \dd{n} = \tdd{n} = \left|\begin{matrix}
 1 & \widetilde{x}_{0} & (\widetilde{x}_{0})^{2} & \cdots & (\widetilde{x}_{0})^{n} \\
 1 & \widetilde{x}_{1} & (\widetilde{x}_{1})^{2} & \cdots & (\widetilde{x}_{1})^{n} \\
\vdots & \vdots & \vdots & \ddots & \vdots \\
 1 & \widetilde{x}_{n} & (\widetilde{x}_{n})^{2} & \cdots & (\widetilde{x}_{n})^{n} \\
 \end{matrix}\right|,\]
and, similarly,
\[g\cdot \dd{n}\bigl(u,x^i\bigr) = \dd{n}\bigl(\widetilde{u},\widetilde{x}^i\bigr) = \left|\begin{matrix}
 1 & \widetilde{x}_{0} & \cdots & (\widetilde{x}_{0})^{i-1} & \widetilde{u}_0 & (\widetilde{x}_{0})^{i + 1}& \cdots & (\widetilde{x}_{0})^{n} \\
 \vdots & \vdots & \ddots & \vdots & \vdots & \vdots & \ddots & \vdots \\
 \vdots & \vdots & \ddots & \vdots & \vdots & \vdots & \ddots & \vdots \\
 1 & \widetilde{x}_{n} & \cdots & (\widetilde{x}_{n})^{i-1} & \widetilde{u}_n & (\widetilde{x}_{n})^{i + 1}& \cdots & (\widetilde{x}_{n})^{n}
\end{matrix}\right|.\]

To state and prove Lemma~\ref{lem: van_action}, it is also necessary to introduce an analogous notation for replacing
multiple columns of our Vandermonde matrix, but the adjustment is easily demonstrated with the following example:
\[\dd{4}\bigl(\bigl[u,x^2\bigr];\bigl[v,x^4\bigr]\bigr) = \left| \begin{matrix*}
 1 & x_0 & u_0 & x_0^3 & v_0\\
 1 & x_1 & u_1 & x_1^3 & v_1\\
 1 & x_2 & u_2 & x_2^3 & v_2\\
 1 & x_3 & u_3 & x_3^3 & v_3\\
 1 & x_4 & u_4 & x_4^3 & v_4
\end{matrix*} \right|.\]

\begin{Lemma}[\cite{zandra_2019}]\label{lem: van_action}
 \begin{equation}\label{eq: detid}
\bigl( \dd{k}\bigr)\dd{k}\bigl(\bigl[u,x^i\bigr]; \bigl[v,x^j\bigr]\bigr) = \dd{k}(u,x^i)\dd{k}\bigl(v,x^j\bigr) - \dd{k}\bigl(v,x^i\bigr)\dd{k}\bigl(u,x^j\bigr).
 \end{equation}
\end{Lemma}

\begin{proof}
 Let $p(x)$ be the interpolating polynomial for the function $u(x)$ at the points $\left\{ (x_i,u_i) \right\}_{i = 0}^k$ and let $q(x)$ be the interpolating polynomial for the function $v(x)$ at the points $\left\{ (x_i,v_i) \right\}_{i = 0}^k$. We begin by observing that $p(x)$ and $q(x)$ can be given the form
 \begin{gather*}
 p(x) = a_0 + a_1x + \cdots + a_kx^k,\qquad
 q(x) = b_0 + b_1x + \cdots + b_kx^k.
 \end{gather*}
 Define $I(p,i)$ to be the $(k + 1)\times (k + 1)$ identity matrix with the $i$-th column (indexed from~$0$) replaced by the coefficients of~$p$. If we let $V$ denote the standard Vandermonde matrix and $V(u,x^i)$ denote the Vandermonde matrix with the column corresponding to $x^i$ replaced by the column vector $(u_0,\dots, u_k)$, then we have that
 \[V\bigl(u,x^i\bigr) = V I(p,i).\]
 And so
 \begin{gather*}
 \dd{k}\bigl(u,x^i\bigr) = \det(V I(p,i)) = \dd{k}\det(I(p,i)) = \dd{k}a_i.
 \end{gather*}
 Using similar notation, we can obtain
 \[V\bigl(\bigl[u,x^i\bigr];\bigl[v,x^j\bigr]\bigr) = V\,I([p,i];[q,j]).\]
 Thus,
 \begin{gather*}
 \bigl( \dd{k} \bigr)\dd{k}\bigl(\bigl[u,x^i\bigr]; \bigl[v,x^j\bigr]\bigl) = \bigl( \dd{k} \bigr)\det(V\,I([p,i];[q,j]))
 = \bigl( \dd{k} \bigr)^2\det(I([p,i];[q,j]))\\
\hphantom{\bigl( \dd{k} \bigr)\dd{k}\bigl(\bigl[u,x^i\bigr]; \bigl[v,x^j\bigr]\bigl)}{}
= \bigl( \dd{k} \bigr)^2\det\begin{pmatrix}a_i&b_i\\a_j&b_j\end{pmatrix}
 = \dd{k}a_i\dd{k}b_j - \dd{k}b_i\dd{k}a_j\\
\hphantom{\bigl( \dd{k} \bigr)\dd{k}\bigl(\bigl[u,x^i\bigr]; \bigl[v,x^j\bigr]\bigl)}{}
 = \dd{k}\bigl(u,x^i\bigr)\dd{k}\bigl(v,x^j\bigr) - \dd{k}\bigl(v,x^i\bigr)\dd{k}\bigl(u,x^j\bigr).\tag*{\qed}
 \end{gather*}\renewcommand{\qed}{}
\end{proof}

Recall that we defined
\[\prol{u}{k}{k} := k!\frac{\dd{k}\bigl(u,x^k\bigr)}{\dd{k}},\]
where $\dd{k}\bigl(u,x^k\bigr)$ stands for the determinant of the Vandermonde matrix with column $x^k$ replaced with $\vec{u}$. To compute the infinitesimal action in terms of multispace coordinates, we will also need to introduce the notation
\[\ulk{\mu}{k}{\ell}:= k!\frac{\dd{k}\bigl(u,x^\ell\bigl)}{\dd{k}},\]
where the upper parenthetical index is used to keep track of the number of points used to construct the
derivative approximation (in this case, $k + 1$ points), and the lower index tracks which column is
replaced in the numerator of the above equation to obtain the derivative information. We may now compute
the infinitesimal action on the multispace coordinates: Let~$\gamma(\varepsilon)$ be a~curve in~$G$ with
$\gamma(0) = e$, let $(\widetilde{x}_i,\widetilde{u}_i) = \gamma(\varepsilon) \cdot (x_i, u_i)$, and let
$\ulk{\varphi}{k}{[k]}$ denote the infinitesimal%
\[\ulk{\varphi}{k}{[k]} := \dep {\proltd{u}{k}{k}}.\]
\begin{Proposition}\label{prop: discrete_action_computation}
 \begin{equation}\label{eq: discrete_action_computation}
 \ulk{\varphi}{k}{[k]} = \prol{\varphi}{k}{k} - \frac{1}{k!}\sum\limits_{\ell = 0}^k \ulk{\mu}{k}{\ell}\prol{\bigl(\ell x^{\ell -1}\xi\bigr)}{k}{k}.
 \end{equation}
\end{Proposition}

\begin{proof}
 We proceed by unraveling the definition of the transformed coordinates and then carefully keep track of the
 derivative with respect to $\varepsilon$:
 \begin{align*}
 \ulk{\varphi}{k}{[k]} &=\dep {\proltd{u}{k}{k}} = \dep \frac{k!\,\dd{k}\bigl(\widetilde{u},\widetilde{x}^k\bigr)}{\tdd{k}}\\
 &= k! \frac{\bigl( \dd{k} \bigr) \dep \dd{k}\bigl(\widetilde{u},\widetilde{x}^k\bigr) - \dd{k}\bigl(u,x^k\bigr) \dep \tdd{k}}{\bigl( \dd{k} \bigr)^2}\\
 &\overset{\eqref{eq: detid}}{=} k! \frac{\bigl( \dd{k} \bigr)\dd{k}\bigl(\varphi,x^k\bigr) - \sum\limits_{\ell = 1}^{k} \bigl[ \bigl( \dd{k}\bigl(u,x^\ell\bigr) \bigr)\dd{k}\bigl(\ell x^{\ell - 1} \xi,x^{k}\bigr) \bigr]}{\bigl( \dd{k} \bigr)^2} \\
 &\overset{\rm def}{=} \prol{\varphi}{k}{k} - \frac{1}{k!}\sum\limits_{\ell = 1}^k \ulk{\mu}{k}{\ell}\prol{\bigl(\ell x^{\ell -1}\xi\bigr)}{k}{k} .\tag*{\qed}
 \end{align*}\renewcommand{\qed}{}
\end{proof}

Unfortunately, this formula does not resemble the formulas for $\varphi_{[k]}^{(k)}$ given in equations~\eqref{eq:
1d-recursion-smooth}, \eqref{eq: 1d-rec-expansion}, and \eqref{eq: useful-derivative} in any meaningful way.
Indeed, if we were to expand out the definitions of $\ulk{\mu}{k}{\ell}$ and $\xilk{\ell}{\ell - 1}{k}$ to
write them in terms of the multispace coordinates, the resulting formula is a mess of
elementary and homogeneous symmetric polynomials in $(x_0,\dots, x_k)$. This is demonstrated below in the expansion
of $\ulk{\varphi}{3}{[3]}$:
\begin{gather*}
 \ulk{\varphi}{3}{[3]} = \prol{\varphi}{3}{3} - \frac{1}{3!}\sum\limits_{\ell = 0}^3 \ulk{\mu}{3}{\ell}\prol{\bigl(\ell x^{\ell -1}\xi\bigr)}{3}{3}
 = \prol{\varphi}{3}{3} - \frac{1}{3!}\bigl[\ulk{\mu}{3}{1} \xilk{1}{0}{3}+ \ulk{\mu}{3}{2}\xilk{2}{1}{3} + \ulk{\mu}{3}{3}\xilk{3}{2}{3}\bigr]\\
\hphantom{\ulk{\varphi}{3}{[3]}}{}
 = \prol{\varphi}{3}{3} - \frac{1}{3!}\Bigl[ \bigl( 6\prol{u}{1}{1} - 3(x_0 + x_1)\prol{u}{2}{2} + (x_0x_1 + x_0x_2 + x_1x_2)\prol{u}{3}{3}\bigr)\prol{\xi}{3}{3} \\
\hphantom{\ulk{\varphi}{3}{[3]}=}{}
 + \bigl( 3\prol{u}{2}{2} - (x_0 + x_1 + x_2)\prol{u}{3}{3}\bigr)\bigl( 2x_3\prol{\xi}{3}{3} + 6\prol{\xi}{2}{2} \bigr) \\
\hphantom{\ulk{\varphi}{3}{[3]}=}{}
+ \prol{u}{3}{3}\bigl( 3x_3^2\prol{\xi}{3}{3} + 9(x_2 + x_3)\prol{\xi}{2}{2} + 18\prol{\xi}{1}{1}\bigr) \Bigr].
\end{gather*}

Nevertheless, these formulas still satisfy a relationship in multispace space analogous to the recursion relationship \eqref{eq: 1d-recursion-smooth} in smooth space.

\begin{Theorem}\label{thm: miraculous_formula}
 Let $\prol{\varphi}{k}{k}$ and $\ulk{\varphi}{k}{[k]}$ be as in Proposition~{\rm \ref{prop: discrete_action_computation}}. Then the recursion relations for the infinitesimals of a group action on the multispace of one independent variable are given by
 \begin{equation}\label{eq: rec_discrete}
 \ulk{\varphi}{k+1}{[k+1]} = \ddx\bigl[ \ulk{\varphi}{k}{[k]}\bigr] - \prol{u}{k+1}{k+1}\left( \frac{1}{x_{k + 1} - x_0}(S^{k+1} - {\rm id})\bigl[ \prol{\xi}{0}{0} \bigr] \right).
 \end{equation}
\end{Theorem}

\begin{proof}
 \begin{gather*}
 \ulk{\varphi}{k+1}{[k+1]} = \dep {\proltd{u}{k+1}{k+1}}
 = (k + 1)\dep\left[ \frac{(S - {\rm id})\left[ \proltd{u}{k}{k}\right]}{\widetilde{x}_{k+1} - \widetilde{x}_{0}} \right]\\
\hphantom{\ulk{\varphi}{k+1}{[k+1]}}{}
= (k + 1)\frac{(x_{k + 1}- x_{0}) (S - {\rm id})\bigl[\dep \proltd{u}{k}{k} \bigr] - (S - {\rm id})\bigl[ \prol{u}{k}{k}\bigr]\dep\bigl[ (\widetilde{x}_{k+1} - \widetilde{x}_0) \bigr]}{(x_{k + 1}- x_{0})^2}\\
\hphantom{\ulk{\varphi}{k+1}{[k+1]}}{}
 = (k + 1)\frac{(S - {\rm id})\bigl[\ulk{\varphi}{k}{[k]} \bigr]}{x_{k + 1}- x_{0}} -
 (k + 1)\frac{(S - {\rm id})\bigl[ \prol{u}{k}{k}\bigr] (S^{k + 1} - {\rm id})\bigl[ \prol{\xi}{0}{0} \bigr]}{(x_{k + 1}- x_{0})^2}\\
\hphantom{\ulk{\varphi}{k+1}{[k+1]}}{}
= \ddx\bigl[ \ulk{\varphi}{k}{[k]} \bigr] - \prol{u}{k+1}{k+1}\left( \frac{1}{x_{k + 1} - x_0}(S^{k+1} - {\rm id})\bigl[ \prol{\xi}{0}{0} \bigr] \right).\tag*{\qed}
 \end{gather*}\renewcommand{\qed}{}
\end{proof}

It is worth taking a moment to consider what Theorem~\ref{thm: miraculous_formula} accomplishes. First of
all, the recursive formula~\eqref{eq: rec_discrete} makes the practical computation of the infinitesimals for the action on the
multispace of curves tractable.
Note that this convenience is directly opposed to the preferred
methodology for jet space prolongation where we often used the expanded formula~\eqref{eq: 1d-rec-expansion}
over the recursive formula~\eqref{eq: 1d-recursion-smooth}.
Even for simple actions, the computation of the infinitesimals in multispace
according to the formulas described in Proposition~\ref{prop: discrete_action_computation} can involve some intense
determinants with highly non-trivial simplifications (see the end of Example~\ref{ex: scaling} for a taste of the
types of rational expressions that appear here). Worse yet, at any given level~$k$, there
are $(2k + 1)$ determinants to compute, and very few of the determinants computed for smaller values of
$k$ may be reused
to compute an infinitesimal for any larger
value of $k$. The formula given in Theorem~\ref{thm: miraculous_formula}, on the other hand, allows us to take
advantage of the properties shown in Proposition~\ref{prop: div-rules} to write down explicit formulas for any
given $\ulk{\varphi}{k}{[k]}$. Then, the computations for any value $k$ will depend heavily on computations
at values $<k$. Practically speaking, this allows for the caching of some terms computed at values
smaller than $k$ to save time in the computation at level~$k$.

\subsection{The coalescent limit}

We are now ready to present the main result of this paper. As previously seen, it is possible to define an
operator on~$M^{(n)}$ that allows for the prolongation of an $n$-pointed curve to an $(n + 1)$-pointed curve.
Moreover,
the infinitesimals for this action have a nearly identical recursive differential relationship to
those for the analogous action on jet space. Since the jet space appears as an embedded
smooth submanifold of multispace, this then suggests that the general recursion relations for prolonging a
group action on jet space are inherited from those on the larger multispace. More specifically, this suggests
that the coalescent limits of the prolongation formulas~\eqref{eq: discrete_action_computation} and \eqref{eq:
rec_discrete} in multispace should be exactly the prolongation formula~\eqref{eq: 1d-recursion-smooth} in the
jet space.

\begin{Theorem}\label{thm: clim-formula}
 \[\clim_{(x_0,\dots,x_n)\rightarrow x} \ulk{\varphi}{k}{[k]} = \varphi_{[k]}.\]
\end{Theorem}

\begin{proof}
 We begin by observing the basic fact that
 \[\clim_{(x_0,\dots,x_n)\rightarrow x}\prol{u}{k+1}{k+1}\left( \frac{1}{x_{k + 1} - x_0}(S^{k+1} - {\rm id})\bigl[ \prol{\xi}{0}{0} \bigr] \right) = u^{(k+1)}\xi^{(1)},\]
 and we proceed by induction. For our base-case, we may compute
 \begin{gather*}
 \clim_{(x_0,x_1)\rightarrow x}\ulk{\varphi}{1}{[1]} \overset{\eqref{eq: rec_discrete}}{=} \clim_{(x_0,x_1)\rightarrow x}\left( \ddx\bigl[ \ulk{\varphi}{0}{[0]} \bigr] - \prol{u}{1}{1}\left( \frac{1}{x_{1} - x_0}(S - {\rm id})\bigl[ \prol{\xi}{0}{0} \bigr] \right) \right)\\
 \hphantom{\clim_{(x_0,x_1)\rightarrow x}\ulk{\varphi}{1}{[1]}}{}
 \ = \frac{{\rm d}}{{\rm d}x}\Bigl[ \clim_{(x_0,x_1)\rightarrow x}\ulk{\varphi}{0}{[0]}\Bigr] - u^{(1)}\xi^{(1)}
 = \frac{{\rm d}}{{\rm d}x}\bigl[ \varphi \bigr] - u^{(1)}\xi^{(1)}
 \overset{\eqref{eq: 1d-recursion-smooth}}{=} \varphi_{[1]}.
 \end{gather*}
 Then, for our induction, if we assume $\displaystyle\clim\limits_{(x_0,\dots, x_{k})}\ulk{\varphi}{k}{[k]} = \varphi_{[k]}$, we obtain
 \begin{gather*}
 \clim_{(x_0,\dots,x_{k + 1})\rightarrow x}\ulk{\varphi}{k + 1}{[k + 1]}
 \overset{\eqref{eq: rec_discrete}}{=} \clim_{(x_0,\dots,x_{k + 1})\rightarrow x}\left( \ddx\bigl[ \ulk{\varphi}{k}{[k]} \bigl] - \prol{u}{k+1}{k+1}\left( \frac{1}{x_{k+1} - x_0}(S^{k + 1} - {\rm id})\bigl[ \prol{\xi}{0}{0} \bigr] \right) \right)\\
 \hphantom{\clim_{(x_0,\dots,x_{k + 1})\rightarrow x}\ulk{\varphi}{k + 1}{[k + 1]}}{}
 \ = \frac{{\rm d}}{{\rm d}x}\Bigl[ \clim_{(x_0,\dots,x_{k + 1})\rightarrow x}\ulk{\varphi}{k}{[k]}\Bigr] - u^{(k+1)}\xi^{(1)}\\
\hphantom{\clim_{(x_0,\dots,x_{k + 1})\rightarrow x}\ulk{\varphi}{k + 1}{[k + 1]}}{}
 \ = \frac{{\rm d}}{{\rm d}x}\bigl[ \varphi_{[k]} \bigr] - u^{(k + 1)}\xi^{(1)}
 \overset{\eqref{eq: 1d-recursion-smooth}}{=} \varphi_{[k + 1]},
 \end{gather*}
 as desired.
\end{proof}

\begin{Corollary}
 \[\clim_{(x_0,\dots,x_{k + 1})\rightarrow x}\ulk{\varphi}{k+1}{[k+1]} = \frac{{\rm d}}{{\rm d}x}\left[ \varphi_{[k]} \right] - u^{(k+1)}\frac{{\rm d}}{{\rm d}x}\xi.\]
\end{Corollary}

Importantly, Theorem~\ref{thm: clim-formula} tells us that the $k$-th order multispace
prolongation of the action on some curve $(x,u)$ sampled at
points $\left\{ x_0,\dots,x_n \right\}$ yields the equivalence class containing the $k$-th order jet at
each of these points. Put more concisely, multicontact equivalent curves behave well under
multispace prolongation of the group action.
This fact is not obvious from the definitions. Recall that the lift from $M^{(n)}$ to
$M^{(n + 1)}$ is not necessarily unique, so it is entirely possible that the action could lift different
sections of the curve into disjoint equivalence classes. Based on the construction of this space, there is
no obvious reason to suspect that this splitting would happen, but the details still needed to be checked
rigorously.

\subsection{Examples}

With the theory out of the way, it is time to consider a couple of examples for computing the infinitesimals of a given group action.

\begin{Example}\label{ex: scaling}
Consider the elementary group action of $\mathbb{R}^*$ on $\mathbb{R}^2$ given by
\[(\widetilde{x}, \widetilde{u}) = \bigl(\lambda^{-1}x, \lambda u\bigr).\]
In the jet case, the infinitesimals are given by
\[\xi = -x,\qquad \varphi = u \qquad\text{and}\qquad \varphi_{[k]} = (k + 1)u^{(k)}.\]
Computing for the multispace case, we find that
\[
\prol{\xi}{0}{0} = -x_{0},\qquad \prol{\xi}{1}{1} = -1\qquad\text{and}\qquad \prol{\xi}{k}{k} = 0\quad \forall k \geq 2.
\]
We also easily see that $\prol{\varphi}{k}{k} = \prol{u}{k}{k}$ for all $k$, so, using equation~\eqref{eq: rec_discrete}, we obtain
\begin{gather*}
 \ulk{\varphi}{1}{[1]} = \prol{u}{1}{1}-\prol{u}{1}{1}\left( \frac{S-{\rm id}}{x_{1}-x_{0}}[-x_0] \right) = 2 \prol{u}{1}{1},\\
 \ulk{\varphi}{2}{[2]} = 2\,\prol{u}{2}{2}-\prol{u}{2}{2}\left( \frac{S^{2}-{\rm id}}{x_{2}-x_{0}}[-x_0] \right) = 3 \prol{u}{2}{2},\\
\cdots\cdots\cdots\cdots\cdots\cdots\cdots\cdots\cdots\cdots\cdots\cdots\cdots\cdots \\
 \ulk{\varphi}{k}{[k]} = (k + 1) \prol{u}{k}{k}.
\end{gather*}
This is pretty nice, but it is worth noting here that without the recursion relation given in Theorem~\ref{thm: miraculous_formula}, the computations are less kind. Even computing $\ulk{\varphi}{2}{[2]}$ using our original formula established in Proposition~\ref{prop: discrete_action_computation} is mildly troublesome when you keep track of all the notation:
\begin{gather*}
 \ulk{\varphi}{2}{[2]} = \prol{\varphi}{2}{2} - \frac{1}{2!}\sum\limits_{\ell = 0}^2 \ulk{\mu}{2}{\ell}\prol{\bigl(\ell x^{\ell -1}\xi\bigr)}{2}{2}\\
\hphantom{\ulk{\varphi}{2}{[2]}}{}
= \prol{u}{2}{2} - \frac{1}{2!}\left( \frac{2 u_{1} x_{2}^{2}-2 x_{1}^{2} u_{2}-2 u_{0} x_{2}^{2}+2 x_{0}^{2} u_{2}+2 u_{0} x_{1}^{2}-2 x_{0}^{2} u_{1}}{-x_{0}^{2} x_{1}+x_{0}^{2} x_{2}+x_{0} x_{1}^{2}-x_{0} x_{2}^{2}-x_{1}^{2} x_{2}+x_{1} x_{2}^{2}} \right)(0) \\
\hphantom{\ulk{\varphi}{2}{[2]}=}{}
- \frac{1}{2!}\prol{u}{2}{2}\left( \frac{4 x_{0}^{2} x_{1}-4 x_{0}^{2} x_{2}-4 x_{0} x_{1}^{2}+4 x_{0} x_{2}^{2}+4 x_{1}^{2} x_{2}-4 x_{1} x_{2}^{2}}{-x_{0}^{2} x_{1}+x_{0}^{2} x_{2}+x_{0} x_{1}^{2}-x_{0} x_{2}^{2}-x_{1}^{2} x_{2}+x_{1} x_{2}^{2}} \right)\\
\hphantom{\ulk{\varphi}{2}{[2]}}{}
= \prol{u}{2}{2} + 2\prol{u}{2}{2} = 3\prol{u}{2}{2}.
\end{gather*}
\end{Example}

\begin{Example}
Let us now consider another quintessential example: rotation in the plane. This is the action of ${\rm SO}(2)$ on $\mathbb{R}^2$ given by
\[\begin{pmatrix}
 \widetilde{x}\\\widetilde{u}
\end{pmatrix} = \begin{pmatrix}
 \cos(\theta) & -\sin(\theta)\\ \sin(\theta) & \cos(\theta)
\end{pmatrix}\begin{pmatrix}
 x\\u
\end{pmatrix}.\]

In the jet case, the first few infinitesimals are
\begin{gather*}
\xi = -u,\qquad \varphi = x,\qquad\varphi_{[1]} = 1 + \bigl( u^{(1)} \bigr)^2,\\
 \varphi_{[2]} = 3 u^{(1)}u^{(2)},\qquad \varphi_{[3]} = 3\bigl( u^{(2)} \bigr)^2 + 4 u^{(1)}u^{(3)}.
\end{gather*}

For the multispace case, we obtain $\prol{\xi}{k}{k} = -\prol{u}{k}{k}$, and
\[\prol{\varphi}{0}{0} = -x_{0},\qquad \prol{\varphi}{1}{1} = -1 \qquad\text{and}\qquad \prol{\varphi}{k}{k} = 0 \quad \forall k \geq 2.\]

The first few prolonged infinitesimals can now be computed using our recursion relations~\eqref{eq: rec_discrete}
\begin{gather*}
 \ulk{\varphi}{1}{[1]} = 1-\prol{u}{1}{1}\left( \frac{S-{\rm id}}{x_{1}-x_{0}}[-u_0] \right) = 1 + \bigl( \prol{u}{1}{1} \bigr)^2,\\
 \ulk{\varphi}{2}{[2]} = \prol{u}{2}{2}\prol{u}{1}{1} + S\bigl[\prol{u}{1}{1}\bigr] \prol{u}{2}{2} - \prol{u}{2}{2}\left( \frac{S^2-{\rm id}}{x_{2}-x_{0}}[-u_0] \right),
\end{gather*}
and
\begin{gather*}
 \ulk{\varphi}{3}{[3]} = \ddx\bigl[ \ulk{\varphi}{2}{[2]} \bigr] - \prol{u}{3}{3}\left( \frac{1}{x_3-x_0}\bigl( S^3-{\rm id} \bigr)\bigl[ \prol{\xi}{0}{0} \bigr] \right)\\
\hphantom{\ulk{\varphi}{3}{[3]}}{}
 =\frac{S - {\rm id}}{x_3 - x_0}\bigl[ \ulk{\varphi}{2}{[2]} \bigr] - \prol{u}{3}{3}\left( \frac{1}{x_3-x_0}\bigl( S^3-{\rm id} \bigr)\bigl[ \prol{\xi}{0}{0} \bigr] \right)\\
\hphantom{\ulk{\varphi}{3}{[3]}}{}
 = \prol{u}{3}{3}\prol{u}{1}{1} + S\bigl[ \prol{u}{2}{2} \bigr]\left( \frac{3}{2} \right)\left( \frac{x_2 - x_0}{x_3 - x_0} \right)\prol{u}{2}{2} \\
 \hphantom{\ulk{\varphi}{3}{[3]}=}{}
+ \left( \frac{3}{2} \right)\left( \frac{x_3 - x_1}{x_3 - x_0} \right)S\bigl[ \prol{u}{2}{2} \bigr]\prol{u}{2}{2} + S\bigl[ \prol{u}{1}{1} \bigr]\prol{u}{3}{3}- \prol{u}{3}{3}\left( \frac{S^2-{\rm id}}{x_2-x_0} [ -u_0 ] \right)\\
 \hphantom{\ulk{\varphi}{3}{[3]}=}{}
 + S\bigl[ \prol{u}{2}{2} \bigr] \left( \frac{3(S-{\rm id})}{x_3 - x_0}\left[ \frac{S^2 - {\rm id}}{x_2 - x_0} [ -u_0 ] \right]\right)
 - \prol{u}{3}{3}\left( \frac{S^3-{\rm id}}{x_{3}-x_{0}}[-u_0] \right).
\end{gather*}
From this example, it is fairly easy to see that
\begin{gather*}
 \clim_{(x_0,\dots,x_n)\rightarrow x}\ulk{\varphi}{1}{[1]} = 1 + \bigl( u^{(1)} \bigr)^2,\qquad
 \clim_{(x_0,\dots,x_n)\rightarrow x}\ulk{\varphi}{2}{[2]} = 3 u^{(1)}u^{(2)}.
\end{gather*}
The last limit is less trivial, but, thanks to Theorem~\ref{thm: clim-formula}, we know that
\[\clim_{(x_0,\dots,x_n)\rightarrow x}\ulk{\varphi}{3}{[3]} = 3\bigl( u^{(2)} \bigr)^2 + 4\,u^{(1)}u^{(3)}.\]
This is a bit easier to see in the special case of an evenly spaced lattice:
\begin{gather*}
 \ulk{\varphi}{3}{[3]} = \prol{u}{3}{3}\prol{u}{1}{1} + 2 S\bigl[ \prol{u}{2}{2} \bigr]\prol{u}{2}{2} + S\bigl[ \prol{u}{1}{1} \bigr]\prol{u}{3}{3}- \prol{u}{3}{3}\left( \frac{S^2-{\rm id}}{2h} [ -u_0 ] \right)\\
\hphantom{\ulk{\varphi}{3}{[3]} =}{}
 + S\bigl[ \prol{u}{2}{2} \bigr] \left( \frac{3(S-{\rm id})}{3h}\left[ \frac{S^2 - {\rm id}}{2h} [ -u_0 ] \right]\right)
 - \prol{u}{3}{3}\left( \frac{S^3-{\rm id}}{3h}[-u_0] \right).
\end{gather*}
\end{Example}

\section{Conclusion and future work}

At the onset of this paper, we set out to answer three questions with regard to a multispace prolongation of the action of a group on a curve.
We defined an operator $\ddx$ on the multispace that behaves analogously to the behavior of~$\frac{{\rm d}}{{\rm d}x}$ on jet space. With this operator in hand, we then showed that the natural way of defining a~prolongation of the action on multipointed curves behaves well under the coalescent limit, in the sense that it converges to the expected smooth analogue.

There are two natural ways in which we might expand on this work. The first would be to investigate the
existence of Noether-like theorems within multispace. To this end, the first step would be to determine the
appropriate analogue for the characteristic for a~vector field acting on $M^{(n)}$. The natural guess would
be to take $Q = \prol{\varphi}{0}{0} - \prol{u}{1}{1}\prol{\xi}{0}{0}$; however, this choice does not allow
us to create a general analogue to equation~\eqref{eq: useful-derivative} as would be desired in order to prove a~true-to-form Noether theorem (cf.\ \cite{dorodnitsyn2001307, mansfield2019moving1}).\footnote{The characteristic
formulation is necessary for Lagrangian and variational methods, but it is possible
to obtain Noether-type theorems without using these methods. See \cite{dorodnitsyn2013integrals,
dorodnitsyn2014_first_integrals}.}

The other natural extension would be to try to determine whether these formulas still hold in the case of a multispace of more than one independent variable. The tricky thing here would be the definition of the divided difference equations. While it is generally possible to find a~polynomial surface passing through a given set of points, the issue when it comes to defining a~multispace in the way that we have here is determining (i) the uniqueness of the interpolating polynomial and (ii) the existence of the coalescent limit. Mar\'i Beffa and Mansfield managed to construct such a multispace that accomplishes both of these in \cite{mari2018discrete}, and work by Olver in \cite{olver2006multivariate} seems to suggest that a more general multispace could be constructed, but practical computation of a prolonged action on the more general space is easier said than done and is currently under investigation by the author of this paper.

\subsection*{Acknowledgements}
The author would like to thank Dr.\ Jeanne Clelland for her guidance and support throughout the writing of this
paper. In addition, the author would like to thank the referees for their helpful comments and
suggestions which significantly contributed to the clarity of this paper.

\pdfbookmark[1]{References}{ref}
\LastPageEnding


\begin{thebibliography}{99}
\footnotesize\itemsep=0pt

\bibitem{deBoordivided2005}
de~Boor C., Divided differences, \textit{Surv. Approx. Theory} \textbf{1}
 (2005), 46--69, \href{https://arxiv.org/abs/math.CA/0502036}{arXiv:math.CA/0502036}.

\bibitem{dorodnitsyn2001307}
Dorodnitsyn V., Noether-type theorems for difference equations, \href{https://doi.org/10.1016/S0168-9274(00)00041-6}{\textit{Appl.
 Numer. Math.}} \textbf{39} (2001), 307--321.

\bibitem{Dorodnitsyn_2022}
Dorodnitsyn V., Kaptsov E., Invariant finite-difference schemes for plane
 one-dimensional {MHD} flows that preserve conservation laws,
 \href{https://doi.org/10.3390/math10081250}{\textit{Mathematics}} \textbf{10} (2022), 1250, 24~pages, \href{https://arxiv.org/abs/2112.03118}{arXiv:2112.03118}.

\bibitem{dorodnitsyn2013integrals}
Dorodnitsyn V., Kaptsov E., Kozlov R., Winternitz P., First integrals of
 ordinary difference equations beyond {L}agrangian methods,
 \href{https://arxiv.org/abs/1311.1597}{arXiv:1311.1597}.

\bibitem{fels1998moving}
Fels M., Olver P.J., Moving coframes.~{I}.~{A}~practical algorithm,
 \href{https://doi.org/10.1023/A:1005878210297}{\textit{Acta Appl. Math.}} \textbf{51} (1998), 161--213.

\bibitem{fels1999moving}
Fels M., Olver P.J., Moving coframes.~{II}.~{R}egularization and theoretical
 foundations, \href{https://doi.org/10.1023/A:1006195823000}{\textit{Acta Appl. Math.}} \textbf{55} (1999), 127--208.

\bibitem{kellison1975numerical}
Kellison S.G., Fundamentals of numerical analysis, Richard D.~Irwin, Inc.,
 1975.

\bibitem{mansfield2010practical}
Mansfield E.L., A practical guide to the invariant calculus, \textit{Cambridge
 Monogr. Appl. Comput. Math.}, Vol.~26, \href{https://doi.org/10.1017/CBO9780511844621}{Cambridge University Press}, Cambridge,
 2010.

\bibitem{mansfield2019moving1}
Mansfield E.L., Rojo-Echebur\'ua A., Hydon P.E., Peng L., Moving frames and
 {N}oether's finite difference conservation laws~{I}, \href{https://doi.org/10.1093/imatrm/tnz004}{\textit{Trans. Math.
 Appl.}} \textbf{3} (2019), tnz004, 47~pages, \href{https://arxiv.org/abs/1804.00317}{arXiv:1804.00317}.

\bibitem{mansfield2019moving2}
Mansfield E.L., Rojo-Echebur\'ua A., Moving frames and {N}oether's finite
 difference conservation laws~{II}, \href{https://doi.org/10.1093/imatrm/tnz005}{\textit{Trans. Math. Appl.}} \textbf{3}
 (2019), tnz005, 26~pages, \href{https://arxiv.org/abs/1808.03606}{arXiv:1808.03606}.


\bibitem{mari2018discrete}
Mar\'{\i}~Beffa G., Mansfield E.L., Discrete moving frames on lattice varieties
 and lattice-based multispaces, \href{https://doi.org/10.1007/s10208-016-9337-5}{\textit{Found. Comput. Math.}} \textbf{18}
 (2018), 181--247.

\bibitem{olver2000applications}
Olver P.J., Applications of {L}ie groups to differential equations,
 \textit{Grad. Texts in Math.}, Vol.~107, \href{https://doi.org/10.1007/978-1-4684-0274-2}{Springer}, New York, 1986.

\bibitem{olver2001geometric}
Olver P.J., Geometric foundations of numerical algorithms and symmetry,
 \href{https://doi.org/10.1007/s002000000053}{\textit{Appl. Algebra Engrg. Comm. Comput.}} \textbf{11} (2001), 417--436.

\bibitem{olver2001joint}
Olver P.J., Joint invariant signatures, \href{https://doi.org/10.1007/s10208001001}{\textit{Found. Comput. Math.}}
 \textbf{1} (2001), 3--67.

\bibitem{olver2006multivariate}
Olver P.J., On multivariate interpolation, \href{https://doi.org/10.1111/j.1467-9590.2006.00335.x}{\textit{Stud. Appl. Math.}}
 \textbf{116} (2006), 201--240.

\bibitem{Tibshirani2022}
Tibshirani R.J., Divided differences, falling factorials, and discrete splines:
 {A}nother look at trend filtering and related problems, \href{https://doi.org/10.1561/2200000099}{\textit{Found. Trends
 Mach. Learn.}} \textbf{15} (2022), 694--846, \href{https://arxiv.org/abs/2003.03886}{arXiv:2003.03886}.

\bibitem{dorodnitsyn2014_first_integrals}
Winternitz P., Dorodnitsyn V.A., Kaptsov E.I., Kozlov R.V., First integrals of
 difference equations which do not posses a variational formulation,
 \href{https://doi.org/10.1134/s1064562414010360}{\textit{Dokl. Math.}} \textbf{89} (2014), 106--109.

\bibitem{zandra_2019}
Zandra M., Theoretical and numerical topics in the invariant calculus of
 variations, Ph.D.~Thesis, {U}niversity of {K}ent, 2020.

\end{thebibliography}
\end{document}